\documentclass{article}
\pdfoutput=1
\usepackage[british]{babel}

\usepackage[T1]{fontenc}
\usepackage{ccfonts,eulervm}

\usepackage{graphicx}

\usepackage[numbers]{natbib}
\usepackage{url}

\usepackage{amsthm,amsmath,amsfonts, euscript}
\theoremstyle{plain}
 \newtheorem{thm}{Theorem}[section]
 \newtheorem{lem}{Lemma}[section]
 \newtheorem{cor}{Corollary}[section]
\theoremstyle{definition}
 \newtheorem{defn}{Definition}[section]

\renewcommand{\d}{\,\operatorname{d}}
\newcommand{\eps}{\varepsilon}
\newcommand{\half}{\tfrac{1}{2}}
\newcommand{\origin}{\textbf{o}}
\newcommand{\quarter}{\tfrac{1}{4}}
\newcommand{\xx}{\textbf{x}_1}
\newcommand{\yy}{\textbf{x}_2}
\newcommand{\Cell}{\mathcal{C}}
\newcommand{\Expect}[1]{\operatorname{\mathbb{E}}\left[#1\right]}
\newcommand{\Leb}{\operatorname{Leb}}
\newcommand{\Prob}[1]{\operatorname{\mathbb{P}}\left[#1\right]}
\newcommand{\Reals}{\mathbb{R}}

\usepackage[totalwidth=480pt,totalheight=680pt]{geometry}

\numberwithin{equation}{section}  

\begin{document}

\title{Return to the Poissonian City} 

\author{Wilfrid S.~Kendall\footnote{Work supported in part by EPSRC Research Grant  EP/K013939.}\\ {\scriptsize \url{w.s.kendall@warwick.ac.uk}}} 

 \date{\text{ }}

\maketitle

\begin{abstract}
Consider the following random spatial network: in a large disk, 
construct a network using a stationary and isotropic Poisson line process of unit intensity.
Connect pairs of points using the network, with initial / final segments of the connecting path formed by travelling off the network in the opposite direction to that of the destination / source.
Suppose further that connections are established using ``near-geodesics'', constructed between pairs of points using the perimeter of the cell containing these two points and formed using only the Poisson lines not separating them.
If each pair of points generates an infinitesimal amount of traffic divided equally between the two connecting near-geodesics, and if the Poisson line pattern is conditioned to contain a line through the centre,
then what can be said about the total flow through the centre? 
In earlier work (``Geodesics and flows in a Poissonian city'', \emph{Annals of Applied Probability}, 21(3), 801--842, 2011) 
it was shown that a scaled version of this flow had asymptotic distribution given by the \(4\)-volume of a region
in \(4\)-space, constructed using an improper anisotropic Poisson line process in an infinite planar strip.
Here we construct a more amenable representation in terms of two ``seminal curves'' defined by the improper Poisson line process, 
and establish results which produce a framework for effective simulation from this distribution up to an \(L^1\) error which tends to zero with increasing computational effort.
\end{abstract}



\section{Introduction}\label{sec:introduction}
What can be said about flows in a random network? 
Aldous, McDiarmid, and Scott \cite{AldousMcDiarmidScott-2009} discuss maximum flows achievable on a complete graph with random capacities;
Aldous and Bhamidi \cite{AldousBhamidi-2010} consider the joint distribution of edge-flows in a complete network with independent exponentially distributed edge-lengths. 
But what can be said about flows in a suitable random \emph{spatial} network?
In previous work with Aldous \cite{AldousKendall-2007}, 
it was shown that sparse Poisson line processes could be used to augment minimum-total-length networks connecting fixed sets of points in such a manner that
(a) the total network length is not appreciably increased, but
(b) the average network distance between two randomly chosen points exceeds average Euclidean distance by only a logarithmic excess.
This result debunks an apparently natural network efficiency statistic, but also indicates attractive features of networks formed using Poisson line processes.

The analysis of \cite{AldousKendall-2007} uses the notion of ``near-geodesics'' as noted in the abstract: these are 
paths constructed between pairs of points using the perimeter of the cell containing these two points and formed using only the Poisson lines not separating them.
Follow-up work 
in \cite{Kendall-2011b} introduces the notion of a ``Poissonian city'',
namely a planar disk of radius \(n\), connected by a random pattern of lines from a stationary and isotropic Poisson line process. 
Pairs of point in the disk are connected by near-geodesics, with initial / final segments of the connecting path formed by travelling off the network in the opposite direction to that of the destination / source.
Conditioning on one of the Poisson lines passing through the centre, and supposing that each pair of points in the disc generates an infinitesimal flow shared equally between two near-geodesics derived from the line pattern,
it can be shown that mean flow at the centre is asymptotic to \(2n^3\); moreover the scaled flow has a distribution which converges to a proper non-trivial distributional limit \cite[Section 3]{Kendall-2011b}. Thus the asymptotic flow at the centre of this random spatial network is well-behaved.
The distribution can be realized in terms of the \(4\)-volume of an unbounded region in \(\Reals^4\) determined by an improper anisotropic Poisson line process defined on an infinite strip; 
however it is a challenge to compute directly with this representation. 
In this paper we show how to represent the volume of this region in terms of a pair of monotonic concave curves (``seminal curves''); 
moreover we establish results which demonstrate that a calculation in terms of initial segments of these seminal curves can be used to approximate the \(4\)-volume up to an explicit \(L^1\) error, which can be made as small as desired.

The paper is organized as follows: in the next section, Section \ref{sec:city}, the Poissonian city and the improper line process are defined; 
Section \ref{sec:seminal} describes the representation in terms of seminal curves;
Section \ref{sec:dynamics} discusses the stochastic dynamics of a seminal curve;
and Section \ref{sec:flow} applies this to determine explicit \(L^1\) error bounds. 
The paper concludes with Section \ref{sec:conclusion}, a brief discussion which mentions an open question related to exact simulation.

\section{Traffic in the improper Poissonian city}\label{sec:city}
A ``Poissonian city'' \cite{Kendall-2011b} is a disk of radius \(n\) with connectivity supplied by lines from a unit-intensity stationary and isotropic Poisson line process.
Recall that such a line process has intensity \(\tfrac12 \d r\,\d\theta\), where the (undirected) lines are parametrized by angle \(\theta\in[0,\pi)\)
and signed distance \(r\) from the origin. (The factor \(\frac12\) ensures that the number of hits on a unit segment has unit mean.) 
Traffic flow in the Poissonian city is supplied by so-called ``near-geodesics'', constructed between pairs of points using the perimeter of the cell containing these two points and formed from lines not separating them.
Short Euclidean connections can be added \cite[Section 1.2]{Kendall-2011b} so as to connect up any pair of points whatsoever, whether the points lie on or off the Poisson line pattern. 
Conditioning on a line running through the origin \(\origin\), one can then study the flow through \(\origin\) which results
if each pair of points contributes the same infinitesimal amount of flow divided equally between two alternate near-geodesics \cite[Section 3]{Kendall-2011b}. 
Asymptotics at \(n\to\infty\) are studied using the limit obtained by considering \(x\to x/n\) together with \(y\to y/\sqrt{n}\): 
the result in the limit is an ``improper Poissonian city'' formed from an improper anisotropic Poisson line process observed within an infinite strip of width \(2\).
Coupling and symmetry arguments are used to show the asymptotic mean flow in the centre is \(2 n^3\) (with limiting distribution when scaled accordingly), 
corresponding to a mean flow at the centre of \(2\) in the improper Poissonian city conditioned to have a horizontal line through \(\origin\) \cite[Theorems 5 and 7]{Kendall-2011b}.

In this section we give a direct description of the improper Poisson city and its associated improper anisotropic Poisson line process,
now observed in the whole plane \(\Reals^2\).
We first consider the intensity measure for the improper line process, using a natural choice of coordinates,
namely the heights of intercepts on the two boundary lines of the infinite strip.
\begin{defn}\label{def:improper}
 Consider the lines in \(\Reals^2\) which are not vertical (which is to say, not parallel to the \(y\)-axis), 
and parametrize these lines by their intercepts \(y_\pm\) on the \(x=\pm1\) axes.
Denote by  \(\Pi_\infty\) the improper anisotropic Poisson line process, composed only of non-vertical lines, 
whose intensity measure \(\nu\) is given in the coordinates \(y_-\) and \(y_+\) by
\begin{equation}\label{eq:improper1}
 \d\nu \quad=\quad \frac{\d y_- \d y_+}{4}\,.
\end{equation}
\end{defn}

Limiting and coupling arguments show that \(\Pi_\infty\) arises as the limiting line process on the infinite strip for large \(n\)
if a Poissonian city within a disc of radius \(n\) is subject to inhomogeneous scaling
\(x\to x/n\) together with \(y\to y/\sqrt{n}\). 
The factor \(\tfrac14\) arises 
(a) from the factor \(\tfrac12\) in the formula for the intensity measure of the Poisson line process of unit intensity given above, 
(b) from the choice of coordinates determined by the two boundary lines of the strip, which are separated by distance \(2\) 
(contrast the expressions for \(\nu\) in other coordinate systems discussed at the start of Section \ref{sec:dynamics} below).

This line process is improper only in the sense of possessing a dense infinity of nearly vertical lines:
once one removes from the line pattern all lines with absolute slope greater than a fixed constant, then the result is locally finite.

\begin{figure}
\includegraphics[height=2in]{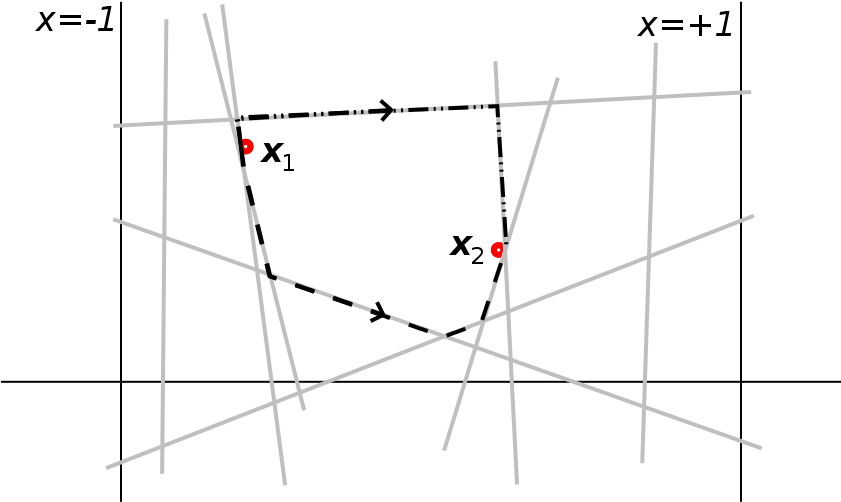}
\centering
 \caption{Indicative illustration of the construction of near-geodesics between points \(\xx=(x_1,y_1)\)  and \(\yy=(x_2,y_2)\) in the improper Poissonian city. 
The two broken lines indicate the pair of near-geodesics between the two points.
Necessarily this indicative illustration omits most of the dense countably infinite set of near-vertical lines contained in \(\Pi_\infty\);
it also omits all lines separating \(\xx\) and \(\yy\).}
\label{fig:city}
\end{figure}

It is immediate from \eqref{eq:improper1} that \(\Pi_\infty\) is statistically invariant under translations, shears along the \(y\)-axis,
and reflections in \(x\) and \(y\) axes.
Calculations also show its statistical invariance under symmetries of the form \(y\to c y\) together with \(x\to c^2 x\) for \(c\neq0\)
(these symmetries are exploited in \cite[Section 3]{Kendall-2011b}).

Following \cite{AldousKendall-2007,Kendall-2011b}, 
we use \(\Pi_\infty\) to construct paths between distinct points \(\xx\) and \(\yy\) in \(\Reals^2\)
which can be thought of as ``near-geodesics'' in the network supplied by \(\Pi_\infty\), 
and correspond to near-geodesics in the original Poissonian cities under the coupling arguments referred to above. 
This construction is illustrated in Figure \ref{fig:city}.
 
\begin{defn}\label{def:near-geodesic}
Fix \(\xx\), \(\yy\in\Reals^2\) and
consider the tessellation generated by all the lines of \(\Pi_\infty\) which do not separate \(\xx\) and \(\yy\). 
Let \(\Cell(\xx,\yy)\) be the (open) tessellation cell
whose closure is the intersection of all the closed half-planes that are bounded by lines of \(\Pi_\infty\) and that
contain both \(\xx\) and \(\yy\). 
The pair of \emph{near-geodesics} between \(\xx\) and \(\yy\) is given by the two connected components obtained by removing \(\xx\) and \(\yy\)
from the perimeter \(\partial\Cell(\xx,\yy)\). 
\end{defn}

In contrast to the case of \cite{AldousKendall-2007,Kendall-2011b}
(where initial / final segments of the connecting path have to be formed \emph{off} the Poisson line process, by travelling in the opposite direction to that of the destination / source), 
the points \(\xx\) and \(\yy\) belong to the closed set \(\partial\Cell(\xx,\yy)\), since there are
infinitely many nearly vertical lines arbitrarily close to \(\xx\) (respectively \(\yy\)) which do not separate \(\xx\) and \(\yy\).

We use these near-geodesics to define a flow over the whole plane \(\Reals^2\): 
the infinitesimal flow between \(\xx\) and \(\yy\) amounts to the infinitesimal quantity \(\d\xx\d\yy\),
and this is divided equally between the two near-geodesics between \(\xx\) and \(\yy\).
We focus attention on the total amount of flow passing through the origin \(\origin\) that is produced by pairs of points
lying on the infinite vertical strip \(\{\mathbf{x}=(x,y):-1<x<1\}\), when
we condition on there being a horizontal line \(\ell^*\in\Pi_\infty\) which passes through \(\origin\) (so in fact \(\ell^*\) is the \(x\)-axis).
Under this conditioning, the total flow of interest is given by
\begin{equation}\label{eq:totalflow}
T \quad=\quad
\int_{-\infty}^\infty\int_{-1}^{+1} \int_{-\infty}^\infty\int_{-1}^{x_2} 
  \half\mathbb{I}_{[\origin\in\partial\Cell((x_1,y_1),(x_2,y_2))]}
\d x_1\d y_1\d x_2\d y_2\,.
\end{equation}
Here the factor of \(\tfrac12\) allows for the splitting of the flow between the two possible near-geodesics.
Note for future use the Slivynak-Mecke theorem \cite[Example 4.3]{ChiuStoyanKendallMecke-2013}: 
if \(\Pi_\infty\) is so conditioned then \(\Pi_\infty\setminus\{\ell^*\}\) is distributed as the original unconditioned improper anisotropic Poisson line process.
So from henceforth the construction of near-geodesics, as in Definition \ref{def:near-geodesic}, 
is based on \(\Pi_\infty\cup\{\ell^*\}\) rather than \(\Pi_\infty\). 

An interaction between the improper nature of \(\Pi_\infty\) and its statistical symmetries can be used to simplify somewhat the quantity \eqref{eq:totalflow}.
Firstly, if one of \(\xx\) or \(\yy\) lies in the open upper half-plane and the other lies in the open lower half-plane,
then the perimeter \(\partial\Cell(\xx,\yy)\) of the (convex) cell 
will almost surely (in \(\xx\) and \(\yy\)) not contain \(\origin\). 
This is a consequence of the horizontal translation symmetry of the statistics of \(\Pi_\infty\cup\{\ell^*\}\). 
Accordingly we can divide the multiple integral \eqref{eq:totalflow} into two non-zero and identically distributed parts, 
integrating respectively over \(y_1>0\) and \(y_2>0\), and \(y_1<0\) and \(y_2<0\).
We shall see in the next section that the contributions from these two parts are independent.

Suppose \(\xx\) and \(\yy\) both lie in the open right-hand half of the open upper half-plane, namely \(\{(x,y):x>0,y>0\}\). 
Since the improper line process \(\Pi_\infty\) contains infinitely many arbitrarily steep lines with \(x\)-intercepts dense on the \(x\)-axis; 
it follows that the near-geodesics between such \(\xx\) and \(\yy\) cannot pass through \(\origin\), 
and so such configurations cannot contribute to \eqref{eq:totalflow}. 
Similarly, no contribution can be made from configurations in which \(\xx\) and \(\yy\) both lie in the left-hand half of the upper half-plane, 
namely \(\{(x,y):x<0,y>0\}\).

Accordingly, the properties of \eqref{eq:totalflow} will follow from analysis of 
\begin{equation}\label{eq:upperflow}
 F \quad=\quad \int_{Q_+}\int_{Q_-} \half\mathbb{I}_{[\origin\in\partial\Cell(\xx,\yy)]}\d\xx\d\yy\,;
\end{equation}
where \(Q_+=\{(x,y): 0<x<1, y>0\}\) while \(Q_-=\{(x,y): -1<x<0, y>0\}\).
In fact the quantity in \eqref{eq:totalflow} will be the independent sum of two copies of \(F\),
one for the upper and one for the lower half-plane.
In \cite{Kendall-2011b} this representation is used to establish some general properties of the flow at the centre.
However it is desirable to construct a representation of \(F\) more amenable to quantitative arguments and effective approximation.
We will now show how to do this.

\section{Separation and seminal curves}\label{sec:seminal}
We focus on the upper half-plane case, and the \(4\)-volume \(2F\) of the subset \(\mathcal{D}^\text{upper}\subset Q_-\times Q_+\) given by
\begin{equation*}
 \mathcal{D}^\text{upper}\quad=\quad\{(\xx, \yy) \in Q_-\times Q_+: \origin \in \partial\Cell(\xx,\yy)\}\,.
\end{equation*}
Thus \(\mathcal{D}^\text{upper}\) is composed of point-pairs \((\xx, \yy)\in Q_-\times Q_+\) 
such that the line segment connecting \(\xx\) with \(\yy\) is not separated from \(\origin\) by \(\Pi_\infty\).
Note that such separation fails if and only if
no one line \(\ell\in\Pi_\infty\) simultaneously separates \(\xx\) and \(\yy\) from \(\origin\).

Consider a dual construction, using the lines of \(\Pi_{\infty}\), that builds sets of lines of positive and negative slope which could in principle separate the origin \(\origin\)
and some
line segment between some \(\xx\in Q_-\) and some \(\yy\in Q_+\):
\begin{align*}
 \Pi_{\infty,+} \quad&=\quad\{\ell\in\Pi_\infty: \ell\text{ has positive slope}, \ell\text{ intercepts negative \(x\)-axis }\}\,,\\
 \Pi_{\infty,-} \quad&=\quad\{\ell\in\Pi_\infty: \ell\text{ has negative slope}, \ell\text{ intercepts positive \(x\)-axis }\}\,.
\end{align*}
The lines relevant to the case of \(\xx\) and \(\yy\) lying in the lower half-plane lie in \(\Pi_\infty\setminus(\Pi_{\infty,-}\cup\Pi_{\infty,+})\):
hence 
(as mentioned in Section \ref{sec:city})
the total flow \eqref{eq:totalflow} is indeed the sum of two independent copies of the upper half-plane contribution \(2F\), for \(F\) as
specified in \eqref{eq:upperflow}.

Now define the \emph{seminal curves} \(\Gamma_\pm\) as the concave lower envelopes of the unions of lines in \(\Pi_{\infty,\pm}\): for \(s\in(0,1]\),
\begin{align}
 \Gamma_-(-s)\quad&=\quad\inf\{\text{ height of intercept of }\ell\text{ on }x=-s\;:\;\ell\in\Pi_{\infty,-}\}\,, \label{eq:gamma-}\\
 \Gamma_+(s)\quad&=\quad\inf\{\text{ height of intercept of }\ell\text{ on }x=s\;:\;\ell\in\Pi_{\infty,+}\}\,.   \label{eq:gamma+}
\end{align}
These are illustrated in Figure \ref{fig:seminal}.
It is immediate that both curves are concave and continuous, and that \(\Gamma_-\) is strictly monotonically decreasing, \(\Gamma_+\) is strictly monotonically increasing.
Therefore the inverses \(\Gamma_-^{-1}(\eps)\), \(\Gamma_+^{-1}(\eps)\) are well-defined for \(0<\eps\leq\Gamma_-(-1)\), \(0<\eps\leq\Gamma_+(1)\) respectively:
it is convenient to adopt the convention that \(\Gamma_\pm^{-1}(\eps)=\pm1\) for \(\eps>\Gamma_\pm(\pm1)\).

\begin{figure}
\includegraphics[height=2in]{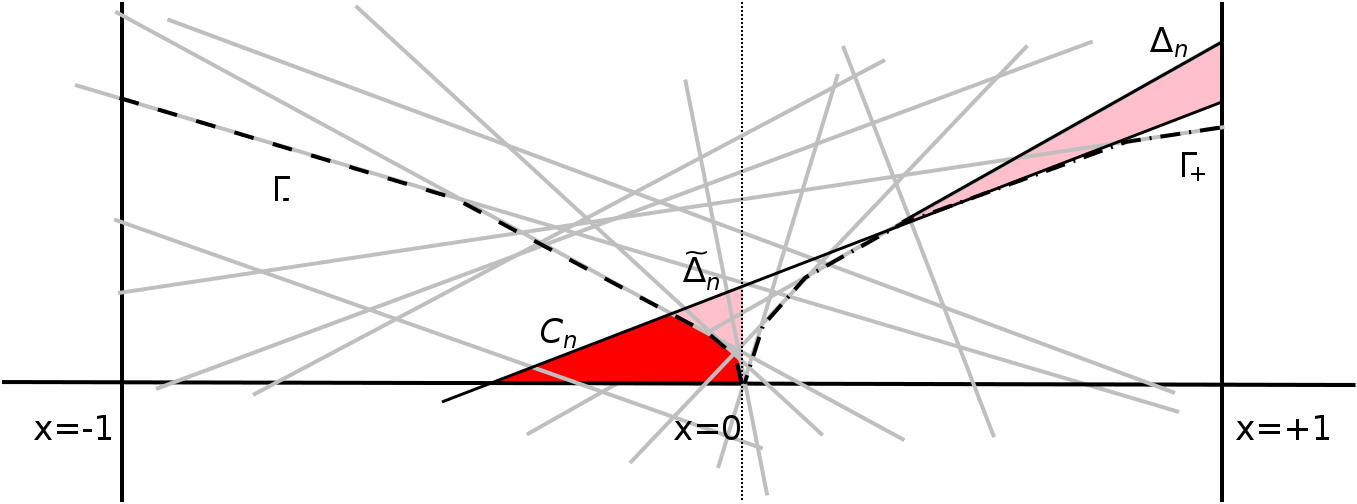}
\centering
 \caption{The two seminal curves \(\Gamma_-\) and \(\Gamma_+\), and the regions \(\Delta_n=\Delta^+_n\), \(C_n=C^+_n\) and \(\widetilde{\Delta}_n=\widetilde{\Delta}^+_n\).
Note that \(\Delta_n\) and \(\widetilde{\Delta}_n\) are triangular regions determined using only lines that are components of \(\Gamma_+\), and the \(x\)-axis, and the \(x=1\) axis. 
The region \(C_n\) is contained in the triangular region \(\widetilde{\Delta}_n\), and uses lines that are components either of \(\Gamma_+\) or of \(\Gamma_-\), as well as the \(x\)-axis.
Note that in fact \(\Gamma_\pm\) have vertical asymptotes at \(0\).}
\label{fig:seminal}
\end{figure}

A simple lower bound for the quantity \eqref{eq:upperflow} arises from the observation that
\begin{equation}\label{eq:area0}
 \left\{(x,y)\in Q_-: 0<y<\Gamma_-(x)\right\}\times\left\{(x,y)\in Q_+: 0<y<\Gamma_+(x)\right\}\;\subset\; \mathcal{D}^\text{upper}\,.
\end{equation}
It may be deduced that
\begin{equation}\label{eq:approx0}
 \left(\int_{-1}^0 \Gamma_-(s)\d s\right)\times\left(\int_0^1 \Gamma_+(s)\d s\right)
\quad<\quad \Leb_4(\mathcal{D}^\text{upper})\quad=\quad 2F\,.
\end{equation}
Evidently it is feasible to approximate both \(\int_{-1}^0 \Gamma_-(s)\d s\) and \(\int_0^1 \Gamma_+(s)\d s\) to within
an additive absolute error of \(\eps>0\) using only finitely many lines from \(\Pi_\infty\), 
namely the lines involved in the initial segments \((\Gamma_-(s):-1\leq s\leq \Gamma_-^{-1}(\eps))\) and \((\Gamma_+(s):\Gamma_-^{-1}(\eps))\leq s \leq 1)\). 
We will see in Section \ref{sec:flow} how this leads to effective approximation.

We therefore turn attention to the difference between the two sides of \eqref{eq:approx0}, 
equivalently the \(4\)-volume of the difference between the two regions in \eqref{eq:area0}.
The difference region splits into two disjoint parts whose definitions are related by the mirror symmetry around the \(y\)-axis. 
First observe that if \(\xx=(x_1,y_1)\in Q_-\) lies above \(\Gamma_-\),
and \(\yy\in Q_+\) lies above \(\Gamma_+\), 
then the line segment connecting \(\xx\) with \(\yy\) is separated from \(\origin\). 
Indeed we can use any line realizing the infimum in the definition of \(\Gamma_-(x_1)\) 
(or any analogous line realizing the infimum in the definition of \(\Gamma_+(x_2)\)).
So we can fix attention on the case when \(\yy=(x_2,y_2)\) lies above \(\Gamma_+\) while \(\xx=(x_1,y_1)\) lies below \(\Gamma_-\),
and use mirror symmetry to deal with the opposite case. 
Consider the lines \(\ell_0\), \(\ell_1\), \(\ell_2\), \ldots of \(\Pi_{\infty,+}\) which are components of \((\Gamma_+(s):0<s\leq1)\), 
enumerated according to the increasing heights of their intercepts on \(x=1\).  
Then \(\origin\in\partial\Cell(\xx,\yy)\) if and only if any \(\ell_n\) lying below \(\yy\) has to pass above \(\xx\).
As a consequence of the infimum-based definition \eqref{eq:gamma+} of \(\Gamma_+\)
and of the concavity of \(\Gamma_+\), 
\(\ell_{n+1}\) must intersect \(\ell_n\), 
and it must do so at a larger \(x\)-coordinate than where it intersects \(\Gamma_+\).
Using concavity and monotonicity of \(\Gamma_+\), it may be deduced that the intercept of \(\ell_n\) on the \(x=x_1<0\) axis must be decreasing in \(n\).  
Let \(n(\yy)\) be the largest \(n\) such that \(\ell_n\) lies below \(\yy\): 
then the set \(C_{n(\yy)}^+\) of \(\xx\) with \(\origin\in\partial\Cell(\xx,\yy)\) is
exactly the set of those points in \(Q_-\) which lie below \(\Gamma_-\) and also below \(\ell_{n(\yy)}\).
Let \(\Delta^+_n\) be the triangle formed by \(\ell_n\), \(\ell_{n+1}\), and the \(x=1\) axis. Figure \ref{fig:seminal} illustrates the definition of these regions,
as well as the further region \(\widetilde{\Delta}^+_n\) to be defined below.

Let \(C_n^-\) and \(\Delta^-_n\) be the analogous regions for lines that are components of \(\Gamma_-\). 
Evidently the areas of both \(C_n^-\) and \(C_n^+\) for any fixed \(n\) can be approximated to within an additive absolute error of \(\eps>0\),
using only the lines involved in \((\Gamma_-(s):-1\leq s\leq \Gamma_-^{-1}(\eps))\) and \((\Gamma_+(s):\Gamma_-^{-1}(\eps)\leq s \leq 1)\),
and the same is trivially true of the triangles \(\Delta^\pm_n\).

It now follows that we can represent \(F\) in \eqref{eq:upperflow} in a way that lends itself to effective approximation so long as we have a useful
representation of the curves \(\Gamma_\pm\) viewed as continuous piecewise-linear random processes: we summarize this in a theorem:
\begin{thm}\label{thm:full-answer}
Given the analysis below of \(\Gamma_\pm\) as continuous piecewise-linear random processes,
\begin{multline}\label{eq:full-answer}
2F\quad=\quad \int_{Q_+}\int_{Q_-} \mathbb{I}_{[\origin\in\partial\Cell(\xx,\yy)]}\d\xx\d\yy
\quad=\quad
 \left(\int_{0}^1 \Gamma_-(-s)\d s\right)\times\left(\int_{0}^1 \Gamma_+(s)\d s\right)\\
+ \sum_{n=0}^\infty \Leb_2(C^+_n) \Leb_2(\Delta^+_n)
+ \sum_{n=0}^\infty \Leb_2(C^-_n) \Leb_2(\Delta^-_n)
\end{multline} 
enables an effective computation of the left-hand side \(2F\): 
indeed, finite truncations of the convergent infinite sums use calculations based on only finitely many of the lines involved in the constructions of \(\Gamma_\pm\).
\end{thm}
\begin{proof}
Using the calculations of \cite[Section 3]{Kendall-2011b}, we can deduce that \(\Expect{F}<\infty\) and therefore that the infinite sums of non-negative terms on the right-hand side are convergent.
By the above arguments, given the subsequent stochastic analysis of \(\Gamma_\pm\) then we may approximate \(\int_{0}^1 \Gamma_-(-s)\d s\)
to within an additive absolute error of \(\tfrac{1}{\Gamma_+(1)}\sqrt{\eps/2}\) and \(\int_{0}^1 \Gamma_+(s)\d s\) to within an additive absolute error of \(\tfrac{1}{\Gamma_-(-1)}\sqrt{\eps/2}\).
Since \(\int_{0}^1 \Gamma_-(-s)\d s<\Gamma_-(-1)\) and \(\int_{0}^1 \Gamma_+(s)\d s<\Gamma_+(1)\), it follows that the 
product of integrals on the right-hand side of Equation \eqref{eq:full-answer} can be approximated to within an additive absolute error of \(\eps/2\).
Moreover 
we can choose to approximate
each term \(\Leb_2(C^\pm_n)\) in the two infinite sums 
to within an additive absolute error of 
\(\tfrac{1}{\Leb_2(\Delta^\pm_n)}2^{-n-2}\eps\). 
Accordingly the entire expression can be approximated to within an additive absolute error of \(\eps\).
While this approximation uses all of \(\Gamma_\pm\), we may truncate the absolutely convergent sums as required to produce an approximation of any required accuracy using only finitely many of these lines.
\end{proof}

In the remainder of this paper we improve on this result by showing that we can provide an explicit \(L^1\)-approximation, by bounding the mean tails of the infinite sums in \eqref{eq:full-answer}.
To prepare for this, consider the region \(C_n^+\).
We can produce a simple triangular approximation region as follows: 
for each \(n\geq0\) 
let \(\widetilde{\Delta}^+_n\) be the triangle formed by \(\ell_n\), the \(x\)-axis, and the \(x=0\) axis; note that this region contains \(C^+_n\). 
Again this region is illustrated in Figure \ref{fig:seminal}.
We can then replace the regions involved in the tails of the infinite sums in \eqref{eq:full-answer}. For example: 
\begin{equation}\label{eq:correction-region}
 \bigcup_{n=N}^\infty \left(C^+_n\times\Delta^+_n\right)\quad\subseteq\quad
 \bigcup_{n=N}^\infty \left(\widetilde{\Delta}^+_n\times\Delta^+_n\right)\,.
\end{equation}
Note that the approximating sets \(\widetilde{\Delta}^+_n\) are now formed entirely from lines in \(\Pi_{\infty,+}\). A similar argument 
applies for the sum involving \(C_n^-\) rather than \(C_n^+\), resulting in an approximating tail using regions formed entirely from lines in \(\Pi_{\infty,-}\),
and therefore the two corrections are independent. If we can obtain \emph{a priori} bounds for the two correction regions, then we have an effective
truncated approximation for \eqref{eq:full-answer}, namely
\begin{equation}\label{eq:full-approx}
 \left(\int_{0}^1 \Gamma_-(-s)\d s\right)\times\left(\int_{0}^1 \Gamma_+(s)\d s\right)\\
+ \sum_{n=0}^N \Leb_2(C^+_n) \Leb_2(\Delta^+_n)
+ \sum_{n=0}^N \Leb_2(C^-_n) \Leb_2(\Delta^-_n)\,.
\end{equation}
This truncated approximation can then itself be approximated in finitary terms, in the sense of involving use of only a finite number of lines of \(\Pi_\infty\)
obtained from  \((\Gamma_-(s):-1\leq s\leq 1/m_-)\) and \((\Gamma_+(s):1/m_+\leq s \leq 1)\)
for suitable \(m_\pm\).

To complete our analysis of the \(4\)-volume specified in \eqref{eq:correction-region} we now need to determine the dynamics of the random processes \((\Gamma_\pm(s):s\in(0,1])\),
both to show that the computations involved in the representation given by Theorem \ref{thm:full-answer} can be achieved effectively, 
and to obtain explicit control of the mean behaviour of the tails of the infinite sums using the upper bounds
\begin{equation}\label{eq:correction1}
  \sum_{n=N}^\infty \left(\Leb_2(\widetilde{\Delta}^+_n)\times\Leb_2(\Delta^+_n)\right)\,.
\end{equation}

\section{Seminal curve dynamics}\label{sec:dynamics}
To prepare for calculation of seminal curve dynamics, we first compute expressions for the intensity measure \(\nu\) in two different coordinate frames.
Consider first the coordinates arising from intercepts \(y_0\), \(y_s\) on \(x=0\), \(x=s\) for some fixed \(s>0\).
This is a linear transformation of coordinates, resulting in
\begin{equation}\label{eq;vertical}
 \d\nu\quad=\quad\frac{\d y_0\d y_s}{2 s}\,.
\end{equation}
Equation \eqref{eq;vertical} makes it evident that \(\nu\) and thus \(\Pi_\infty\) satisfy the (statistical) symmetry \(y\to c y\), \(x\to c^2x\) for non-zero \(c\).
Secondly, consider new coordinates given by slope \(\sigma\) and intersection \(x\) with a fixed reference line of slope \(\sigma_0\),
and intercepts \(b_0\), \(b_s\) on \(x=0\), \(x=s\) for some fixed \(s>0\). 
In \(y_0\), \(y_s\) coordinates we find
\begin{align*}
 (s-x) b_0 + x b_s \quad&=\quad (s-x) y_0 + x y_s \,,\\
 y_0 + \sigma x    \quad&=\quad b_0 + \sigma_0 x\,.
\end{align*}
Now examine \(\nu_{\infty,+}\) obtained as the intensity measure of \(\Pi_{\infty,+}\).
We get different answers for the regions in which \(\sigma\) is less than 
or greater than \(\sigma_0\): recalling that all lines in \(\Pi_{\infty,+}\)
are of positive slope and intersect the negative part of the \(x\)-axis, calculations yield
\begin{align}\label{eq:slant-}
 \d\nu_{\infty,+}\quad&=\quad \half (\sigma_0-\sigma)\d\sigma\d x\qquad\text{ valid for }0<\sigma<\sigma_0\,,\\
\label{eq:slant+}
 \d\nu_{\infty,+}\quad&=\quad \half (\sigma-\sigma_0)\d\sigma\d x\qquad\text{ valid for }\sigma_0<\sigma<\sigma_0+b_0/x\,.
\end{align}

For convenience, we concentrate attention on \(\Gamma=\Gamma_+\). 
We now calculate the one-point joint distribution of \((\Gamma(s), \Gamma'(s))\) for \(0<s\leq1\), bearing in mind that the two-sided
derivative \(\Gamma'(s)\) exists only almost surely for each fixed non-zero \(s\):
\begin{lem}\label{lem:1point1}
 For \(s\in(0,1]\),
\begin{align}\label{eq:1point1}
 \Prob{\Gamma(s)>\gamma}\quad&=\quad \exp\left(-\frac{\gamma^2}{4s}\right) \qquad \text{ for } \gamma>0\,,\\
 \mathcal{L}\left(\Gamma'(s)\right) \quad&=\quad \text{Uniform}\left[0, \frac{\Gamma(s)}{s}\right]\,.
\end{align}
In particular, \(\Gamma(s)\) has a Rayleigh\((\sqrt{2s})\) distribution. 
\end{lem}
\begin{proof}
Consider the intensity measure \(\nu\) in \(y_0\), \(y_s\) coordinates, as specified in  \eqref{eq;vertical}.
It follows that the point process of intersections of the \(x=s\) axis with lines from \(\Pi_{\infty,+}\), 
with each intersection marked by the slope of the corresponding line, 
is given by an inhomogeneous Poisson process of points \(0<t_1<t_2<\ldots\), with intensity measure \(\tfrac{1}{2s}t\d t\), 
such that each point \(t_m\) is independently marked by a slope with distribution Uniform\([0,t_m/s]\). 
The result follows immediately.
\end{proof}

These arguments can be extended to determine the two-point joint distribution of
the pair of pairs \((\Gamma(s), \Gamma'(s))\) and \((\Gamma(t), \Gamma'(t))\).
However for the purposes of Theorem \ref{thm:full-answer} we need to understand the dynamical behaviour of the random process \((\Gamma(s):s\in(0,1])\). 
It turns out to be most convenient to study this process in reversed time, so we take \(\Gamma'(s)\) to be continuous from the left and to have 
right limits (``c\`agl\`ad'', to use a French probabilistic acronym).

\begin{thm}\label{thm:dynamics}
 Consider the times of changes of slope of \((\Gamma(s):0<s\leq1)\) in reversed time:
\[
 1 \quad=\quad  S_0 \quad>\quad S_1 \quad>\quad S_2 \quad>\quad \ldots\quad>\quad 0\,.
\]
Using the enumeration of tangent lines \(\ell_0\), \(\ell_1\), \(\ell_2\), \ldots, from Section \ref{sec:seminal}, we find that 
\(\ell_n\) has the slope of \(\Gamma'(s)\) for \(S_n\geq s>S_{n+1}\).
Writing \(Y_n=\Gamma(S_n) - S_n \Gamma^\prime(S_n)\) for the intercept of \(\ell_n\) on the \(y\) axis,
\begin{align}
 \frac{1}{S_{n+1}} \quad&=\quad \frac{1}{S_n} + \frac{4}{Y_n^2} E_{n+1}\,,\label{eq:S}\\
\Gamma^\prime(S_{n+1}) \quad&=\quad \Gamma^\prime(S_n) + \frac{Y_n}{S_{n+1}}\sqrt{U_{n+1}}\,.\label{eq:Gamma'}
\end{align}
where the \(E_n\) have standard Exponential distributions, the \(U_n\) have Uniform\([0,1]\) distributions, and all are independent of each other and
of \(\Gamma(S_0)=\Gamma(1)\) and \(\Gamma^\prime(S_0)=\Gamma'(1)\) (whose joint distribution follows from \eqref{eq:1point1}).
\end{thm}

This yields a dynamical algorithm to simulate \(\Gamma(s)\) segment-by-segment as \(s\) decreases down to \(0\).
This is what is required in order to generate a simulation recipe for the approximation \eqref{eq:full-approx}.

\begin{proof}
Again the proof follows from re-expressing the intensity measure \(\nu\) in new coordinates, this time as given by \eqref{eq:slant+}.
This calculation can be applied to the point process of intersections
of \(\Pi_{\infty,+}\) with a fixed reference line of slope \(\sigma_0\),
and intercepts \(b_0\), \(b_s\) on \(x=0\), \(x=s\) for some fixed \(s>0\). 
Restrict attention to the case when the intercepting line has slope greater than \(\sigma_0\). 
Considering the point process of intercepts with each intersection marked by the slope of the corresponding line, 
the subprocess of intercepts \(0<x_1<x_2<\ldots\) with slope greater than \(\sigma_0\) has intensity measure \(\quarter (b_0/x)^2 \d x\), and each point \(x_m\) is independently marked 
by a slope with density \(2(\sigma-\sigma_0)/(b_0/x)^2\), for \(0<\sigma_0<\sigma<\sigma_0+b_0/x\).
The result follows by calculation.
\end{proof}

These dynamics are ``reverse-time dynamics''. The calculations of \eqref{eq:slant-} could be applied to determine ``forward-time dynamics'': however these are not 
useful for our current purposes.

We now state and prove three corollaries about the behaviour of the system \(((S_n,Y_n):n\geq0)\).
The recursive system \eqref{eq:S} and \eqref{eq:Gamma'} leads to a delightfully simple expression for the intercept process \((Y_{n}:n\geq0)\)
as a perpetuity \cite{Vervaat-1979}:
\begin{cor}\label{cor:perpetuity}
In the notation of Theorem \ref{thm:dynamics},
\begin{equation}\label{eq:perpetuity}
 Y_{n+1}\quad=\quad Y_n (1-\sqrt{U_{n+1}}) \quad=\quad Y_0 \prod_{m=1}^{n+1} (1-\sqrt{U_m})\,.
\end{equation}
In particular, \(\operatorname{limsup}_{n\to\infty} 3^n Y_n\) is a finite random variable, and so almost surely \(Y_n\) converges to zero geometrically fast.
\end{cor}
\begin{proof}
 The perpetuity equation \eqref{eq:perpetuity} can be deduced directly from the expression for \(Y_n\) and \eqref{eq:Gamma'}.
It follows from \(\Expect{1-\sqrt{U_n}}=\tfrac{1}{3}\) that \((3^nY_n:n\geq0)\) is a non-negative martingale, and therefore converges to a non-negative random limit.
\end{proof}

From \eqref{eq:S} we can deduce that
\[
 \frac{Y_{n+1}^2}{S_{n+1}}\quad=\quad
 \frac{Y_{n+1}^2}{S_{n}} + 4\frac{Y_{n+1}^2}{Y_n^2} E_{n+1}
\quad=\quad \frac{Y_{n+1}^2}{Y_n^2}\left(\frac{Y_{n}^2}{S_{n}} + E_{n+1}\right) 
\quad=\quad (1-\sqrt{U_{n+1}})^2\left(\frac{Y_{n}^2}{S_{n}} + E_{n+1}\right)\,.
\]
Consequently we can take conditional expectations and use independence and a Foster-Lyapunov argument 
(\cite[Ch.~15 especially Theorem 15.0.1]{MeynTweedie-1993} or \cite[Theorem 3.1]{RobertsTweedie-1996b})
to reveal the following:
\begin{cor}\label{cor:GE}
 \((Y_n^2/S_n:n\geq0)\) forms a geometrically ergodic Markov chain.
\end{cor}
One further step is useful in understanding the error bound.
\begin{cor}\label{cor:UP}
The mean value \(\Expect{Y_n^3/S_n}\) converges to zero geometrically fast. Indeed:
\begin{equation}\label{eq:decay}
\Expect{\frac{Y_n^3}{S_n}}\quad\leq\quad  \text{constant }\times 3^{-n}\,.
\end{equation}
In particular, \(Y_n^3/S_n\) almost surely converges to zero geometrically fast.
\end{cor}
\begin{proof}
Applying \eqref{eq:S}, 
\begin{multline*}
 \Expect{3^n\frac{Y_n^3}{S_n}}\quad=\quad
 \Expect{3^n\frac{Y_n^3}{Y_{n-1}^3} \left(\frac{Y_{n-1}^3}{S_{n-1}} + 4 Y_{n-1}E_n \right)}\\
\quad=\quad
 \Expect{3^n(1-\sqrt{U_n})^3 \left( \frac{Y_{n-1}^3}{S_{n-1}} + 4  Y_{n-1} E_n\right)}
\quad=\quad
\frac{3}{10} \Expect{3^{n-1}\frac{Y_{n-1}^3}{S_{n-1}}+ 4Y_0}\\
\quad=\quad
\Expect{\left(\frac{3}{10}\right)^n \frac{Y_{0}^3}{S_{0}}+ 4\left(\left(\frac{3}{10}\right)^n+\left(\frac{3}{10}\right)^{n-1}+\dots+\frac{3}{10}\right)Y_0}\\
\quad\leq\quad 
\Expect{\left(\frac{3}{10}\right)^n \frac{Y_{0}^3}{S_{0}}+ \frac{12}{7}Y_0}\,.
\end{multline*}
But \(S_0=1\) while \(Y_0\) has a Rayleigh\((\sqrt{2})\) distribution and therefore has finite moments of all orders.
\end{proof}

\section{Flow in the centre of the city}\label{sec:flow}
From the above work, we can represent the flow at the centre of the city in terms of the seminal curves.
Here we establish an explicit upper bound on the \(L^1\) error that arises if we use only finite portions of the seminal curves. 

Consider the tail-sum \eqref{eq:correction1}, from which we can obtain an \(L^1\) upper bound on the error term. 
This can be expressed in terms of the quantities studied in the dynamical system given by \eqref{eq:S} and \eqref{eq:Gamma'}:
\begin{equation}
 \sum_{n=N}^\infty \Leb_2(\widetilde{\Delta}^+_n) \Leb_2(\Delta^+_n)
\quad=\quad
\quarter \sum_{n=N}^\infty (1-S_{n+1})^2 Y_{n}^2 \left(\frac{\Gamma^\prime(S_{n+1})}{\Gamma^\prime(S_{n})}-1\right)  
\end{equation}
since
\begin{align*}
 \Leb_2(\widetilde{\Delta}^+_n) \quad&=\quad \half Y_n \times Y_n / \Gamma^\prime(S_n)\,,\\
 \Leb_2(\Delta^+_n) \quad&=\quad\half (1-S_{n+1}) \times (\Gamma'(S_{n+1})-\Gamma'(S_n)))(1-S_{n+1})\,.
\end{align*}

We now estimate the \(n^\text{th}\) summand of \eqref{eq:correction1} for any \(n\geq N\),
using the fact that \(0<S_n\leq1\), the above details about the stochastic dynamics, and the fact that \(\Gamma^\prime\) is monotonically decreasing,
\begin{multline*}
 (1-S_{n+1})^2 Y_{n}^2 \left(\frac{\Gamma^\prime(S_{n+1})}{\Gamma^\prime(S_{n})}-1\right) \quad\leq\quad
Y_{n}^2 \left(\frac{\Gamma^\prime(S_{n+1})}{\Gamma^\prime(S_{n})}-1\right)
\quad=\quad
Y_{n}^3 \left(\frac{\sqrt{U_{n+1}}}{\Gamma^\prime(S_{n})S_{n+1}}\right)\\
\quad=\quad
\frac{Y_n^3\sqrt{U_{n+1}}}{\Gamma^\prime(S_n)}\left(\frac{4}{Y_n^2}E_{n+1} + \frac{4}{Y_{n-1}^2}E_{n} +\ldots+ \frac{4}{Y_{N}^2}E_{N+1} + \frac{1}{S_N}\right)
\\
\quad\leq\quad
\frac{4Y_n\sqrt{U_{n+1}}}{\Gamma^\prime(S_N)}
\left(E_{n+1} + 
\left(\frac{Y_n}{Y_{n-1}}\right)^2 E_{n} 
+\ldots+ 
\left(\frac{Y_n}{Y_{N}}\right)^2 E_{N+1}
\right)  
+ \frac{Y_n^3\sqrt{U_{n+1}}}{\Gamma^\prime(S_N)S_N}
\end{multline*}
Now take conditional expectations given \(\Gamma'(S_N)\), \(S_N\), \(Y_N\) and use the independence of the \(E_n\) and the \(Y_n/Y_N\) (for \(n\geq N\)) to convert
the conditional expectations into absolute expectations,
using also the product expression \eqref{eq:perpetuity} for the perpetuity \(Y\),
and the fact that \(\Gamma'(S_N)\geq\Gamma'(S_0)\) for \(N\geq0\):
\begin{multline*}
 \Expect{(1-S_{n+1})^2 Y_{n}^2 \left(\frac{\Gamma^\prime(S_{n+1})}{\Gamma^\prime(S_{n})}-1\right) \;\Big|\; \Gamma'(S_N), S_N, Y_N}
\\\quad\leq\quad
\frac{2}{3}\frac{Y_N}{\Gamma^\prime(S_N)}\left(
4\Expect{ \frac{Y_n}{Y_N}
\left(1 + 
\left(\frac{Y_n}{Y_{n-1}}\right)^2 
+\ldots+ 
\left(\frac{Y_n}{Y_{N}}\right)^2 
\right)}
+ 
\Expect{\left(\frac{Y_n}{Y_N}\right)^3} \frac{Y_N^2 }{S_N}
\right)\\
\quad\leq\quad
\frac{2}{3}\frac{Y_N}{\Gamma^\prime(S_N)}\left(
4\left( 1 + \frac{3}{10} + \ldots + \left(\frac{3}{10}\right)^{n-N}\right)\left(\frac{1}{3}\right)^{n-N} 
+ 
\left(\frac{1}{10}\right)^{n-N}\frac{Y_N^2 }{S_N} 
\right)\\
\quad\leq\quad
\frac{2}{3} \frac{Y_N}{\Gamma^\prime(S_N)}\left(
\frac{40}{7} \left(\frac{1}{3}\right)^{n-N}  + \frac{Y_N^2 }{S_N}  \left(\frac{1}{10}\right)^{n-N}
\right)\,.
\end{multline*}

Hence
\begin{multline}\label{eq:final-expression}
  \Expect{\sum_{n=N}^\infty \Leb_2(\widetilde{\Delta}^+_n) \Leb_2(\Delta^+_n)\;\Big|\; \Gamma'(S_N), S_N, Y_N}
\quad\leq\quad
\frac{5}{\Gamma^\prime(S_N)} \left(\frac{2}{7} Y_N  + \frac{1}{27}\frac{Y_N^3 }{S_N}\right)\\
\quad\leq\quad
\frac{5}{\Gamma^\prime(S_0)} \left(\frac{2}{7} Y_N  + \frac{1}{27}\frac{Y_N^3 }{S_N}\right)\,. 
\end{multline}
Now \(Y_N\) and \(Y_N^3/S_N\) almost surely converge geometrically fast to zero (use Corollaries \ref{cor:perpetuity} and \ref{cor:UP}). 
Hence almost surely the above conditional expectation will tend to zero as \(N\to\infty\).
Moreover we have the following explicit \(L^1\) error bound, converging geometrically fast to zero with \(N\).
\begin{thm}\label{thm:truncation}
The \(L^1\) error of the approximation
 \begin{multline}\label{eq:truncation}
2F \quad=\quad \int_{Q_+}\int_{Q_-} \mathbb{I}_{[\origin\in\partial\Cell(\xx,\yy)]}\d\xx\d\yy
\quad\approx\quad
 \left(\int_{0}^1 \Gamma_-(-s)\d s\right)\times\left(\int_{0}^1 \Gamma_+(s)\d s\right)\\
+ \sum_{n=0}^N \Leb_2(C^+_n) \Leb_2(\Delta^+_n)
+ \sum_{n=0}^N \Leb_2(C^-_n) \Leb_2(\Delta^-_n)
\end{multline}
is bounded above by 
\begin{equation}\label{eq:conditional-error}
\frac{20}{7} \times 3^{-N} + \frac{20}{27}\times 6^{-N}
\end{equation}
\end{thm}
\begin{proof}
 By the previous work, we obtain a bound on the approximation error for \eqref{eq:truncation} by replacing \(N\) for \(N-1\) in the sum of (a) the term
\[
 \frac{5}{\Gamma^\prime(S_N)} \left(\frac{2}{7} Y_N  + \frac{1}{27}\frac{{Y_N}^3 }{S_N}\right)
\]
and (b) the corresponding term  for the left seminal curve \(\Gamma_-\) as opposed to \(\Gamma_+=\Gamma\). We now estimate the quantity in (a).

In the first place, observe that Equation \eqref{eq:Gamma'}, and the fact that \(Y_0\leq\Gamma(1)\), allows us to deduce that
\begin{multline*}
 \Expect{\frac{Y_N^3}{\Gamma'(S_N)S_N}}\quad\leq\quad \Expect{\frac{Y_N^3}{Y_{N-1}\sqrt{U_N}}} \\
\quad=\quad  \Expect{\frac{(1-\sqrt{U_N})^3}{\sqrt{U_N}}} \left(\Expect{(1-\sqrt{U_1})^2}\right)^{N-1} \Expect{Y_0^2}\\
\quad\leq\quad \frac{1}{2} \times 6^{-(N-1)} \Expect{\Gamma(1)^2}
\quad=\quad 2 \times 6^{-(N-1)}\,,
 \end{multline*}
where the last step uses the fact that \(\Gamma(1)\) has a Rayleigh\((\sqrt{2})\) distribution: see \eqref{eq:1point1}. 
In the second place, consider
\begin{multline*}
 \Expect{\frac{Y_N}{\Gamma'(S_N)}} \quad\leq\quad \Expect{\frac{Y_N}{\Gamma'(S_1)}}
\quad\leq\quad 
\Expect{\frac{Y_N S_1}{Y_0 \sqrt{U_1}}}
\quad\leq\quad 
\Expect{\frac{Y_N}{Y_0 \sqrt{U_1}}}\\
\quad=\quad
\left(\Expect{1-\sqrt{U_2}}\right)^{N-1} \Expect{\frac{1-\sqrt{U_1}}{\sqrt{U_1}}}
\quad=\quad
3^{-(N-1)}\,.
\end{multline*}
The result follows by calculating the contribution from (a) and then doubling to account for the contribution from (b).

\end{proof}
A variation on this argument gives a geometrically decaying \emph{conditional} \(L^1\) error bound given \(S_N\),\(Y_N\) and \(\Gamma'(S_N)\)
and their counterparts for the left seminal curve \(\Gamma_-\) as opposed to \(\Gamma_+\). 
However this conditional \(L^1\) error bound is rather inelegant, since the quantities \(C^\pm_n\) in \eqref{eq:truncation} depend on \(S_m\),\(Y_m\) and \(\Gamma'(S_m)\) for \(m\geq N\).

We can summarize these results as follows: the \(4\)-volume given by \eqref{eq:totalflow} can be approximated to any desired accuracy in \(L^1\), 
based on construction of initial segments of the seminal curves \((\Gamma_-(s):-1\leq s\leq 1/m_-)\) and \((\Gamma_+(s):1/m_+\leq s \leq 1)\) 
and their lower half-plane counterparts, for suitable \(m_\pm\),
using Theorem \ref{thm:full-answer}, Lemma \ref{lem:1point1}, Theorem \ref{thm:dynamics} and Theorem \ref{thm:truncation}.

\section{Conclusion}\label{sec:conclusion}
The asymptotic traffic flow in a Poissonian city has been represented as the \(4\)-volume of a stochastic geometric object in \cite{Kendall-2011b}, 
but the object itself (an unbounded region in \(\Reals^4\)) is somewhat intransigent.
The above work shows how to represent the volume in terms of integrals involving the strictly monotonic continuous concave seminal curves \(\Gamma_\pm\),
and furthermore establishes approximations which supply the theory necessary to approximate and effectively simulate the \(4\)-volume with explicit \(L^1\) error.

Work for a future occasion includes investigation of the amount of computational effort required to achieve stage \(N\) approximations corresponding to \eqref{eq:truncation}.
This is a non-trivial task, since account must be taken of the effort required to approximate each of the \(C^\pm_n\) for \(n=0,1,\ldots,N\).

It is natural then to ask whether it might be possible to translate this work into construction of a perfect simulation algorithm. 
For example M{\o}ller's nearly perfect simulation algorithm for conditionally specified models \cite{Moller-1999} (simulating to within floating point error)
was improved by Wilson to an efficient and exactly perfect simulation algorithm using multishift coupling \cite{Wilson-2000b}. 
Certainly Fill and Huber \cite{FillHuber-2010} have shown how to use dominated coupling from the past to generate exact draws from recursive definitions of perpetuities 
(see also work of Blanchet and Sigman \cite{BlanchetSigman-2011});
this is suggestive, since the system for \(\Gamma'(S_n)\) and \(S_n\) (\eqref{eq:S} and \eqref{eq:Gamma'}) is a similar if more complicated recursive construction.
However, in the present case, interest lies in integral quantities derived from \eqref{eq:S} and \eqref{eq:Gamma'}, 
and it is not obvious how to generate a perfect simulation algorithm from the approximate simulation algorithm implied by the results of Theorem \ref{thm:full-answer}, 
Lemma \ref{lem:1point1}, Theorem \ref{thm:dynamics} and Theorem \ref{thm:truncation}. 
The matter of whether or not such a perfect simulation algorithm exists is left as a significant open question for future work.



\bibliographystyle{plainnat}
\bibliography{Kendall_JAP51A}

\vfill

\leftline{\textsc{Department of Statistics, University of Warwick, Coventry CV4 7AL, UK}}
\leftline{\emph{Email address:} \url{w.s.kendall@warwick.ac.uk}}

\end{document}